\documentclass[a4paper,10pt]{amsart}
\usepackage{amsmath,amsthm,amssymb,latexsym,enumerate,color,hyperref}
\usepackage{graphicx}

\numberwithin{equation}{section}

\begin{document}

\newtheorem{thm}{Theorem}[section]
\newtheorem{prop}[thm]{Proposition}
\newtheorem{lem}[thm]{Lemma}
\newtheorem{cor}[thm]{Corollary}
\newtheorem{rem}[thm]{Remark}
\newtheorem*{defn}{Definition}

\newcommand{\DD}{\mathbb{D}}
\newcommand{\NN}{\mathbb{N}}
\newcommand{\ZZ}{\mathbb{Z}}
\newcommand{\QQ}{\mathbb{Q}}
\newcommand{\RR}{\mathbb{R}}
\newcommand{\CC}{\mathbb{C}}
\renewcommand{\SS}{\mathbb{S}}

\renewcommand{\theequation}{\arabic{section}.\arabic{equation}}

\newcommand{\supp}{\mathop{\mathrm{supp}}}    

\newcommand{\re}{\mathop{\mathrm{Re}}}   
\newcommand{\im}{\mathop{\mathrm{Im}}}   
\newcommand{\dist}{\mathop{\mathrm{dist}}}  
\newcommand{\link}{\mathop{\circ\kern-.35em -}}
\newcommand{\spn}{\mathop{\mathrm{span}}}   
\newcommand{\ind}{\mathop{\mathrm{ind}}}   
\newcommand{\rank}{\mathop{\mathrm{rank}}}   
\newcommand{\Fix}{\mathop{\mathrm{Fix}}}   
\newcommand{\codim}{\mathop{\mathrm{codim}}}   
\newcommand{\conv}{\mathop{\mathrm{conv}}}   
\newcommand{\epsi}{\mbox{$\varepsilon$}}
\newcommand{\eps}{\mathchoice{\epsi}{\epsi}
{\mbox{\scriptsize\epsi}}{\mbox{\tiny\epsi}}}
\newcommand{\cl}{\overline}
\newcommand{\pa}{\partial}
\newcommand{\ve}{\varepsilon}
\newcommand{\zi}{\zeta}
\newcommand{\Si}{\Sigma}
\newcommand{\cG}{{\mathcal G}}
\newcommand{\cH}{{\mathcal H}}
\newcommand{\cI}{{\mathcal I}}
\newcommand{\cJ}{{\mathcal J}}
\newcommand{\cK}{{\mathcal K}}
\newcommand{\cL}{{\mathcal L}}
\newcommand{\cN}{{\mathcal N}}
\newcommand{\cR}{{\mathcal R}}
\newcommand{\cS}{{\mathcal S}}
\newcommand{\cT}{{\mathcal T}}
\newcommand{\cU}{{\mathcal U}}
\newcommand{\OM}{\Omega}
\newcommand{\B}{\bullet}
\newcommand{\ol}{\overline}
\newcommand{\ul}{\underline}
\newcommand{\vp}{\varphi}
\newcommand{\AC}{\mathop{\mathrm{AC}}}   
\newcommand{\Lip}{\mathop{\mathrm{Lip}}}   
\newcommand{\es}{\mathop{\mathrm{esssup}}}   
\newcommand{\les}{\mathop{\mathrm{les}}}   
\newcommand{\nid}{\noindent}
\newcommand{\pzr}{\phi^0_R}
\newcommand{\pir}{\phi^\infty_R}
\newcommand{\psr}{\phi^*_R}
\newcommand{\pow}{\frac{N}{N-1}}
\newcommand{\ncl}{\mathop{\mathrm{nc-lim}}}   
\newcommand{\nvl}{\mathop{\mathrm{nv-lim}}}  
\newcommand{\la}{\lambda}
\newcommand{\La}{\Lambda}    
\newcommand{\de}{\delta}    
\newcommand{\fhi}{\varphi} 
\newcommand{\ga}{\gamma}    
\newcommand{\ka}{\kappa}   

\newcommand{\core}{\heartsuit}
\newcommand{\diam}{\mathrm{diam}}

\newcommand{\lan}{\langle}
\newcommand{\ran}{\rangle}
\newcommand{\tr}{\mathop{\mathrm{tr}}}
\newcommand{\diag}{\mathop{\mathrm{diag}}}
\newcommand{\dv}{\mathop{\mathrm{div}}}

\newcommand{\al}{\alpha}
\newcommand{\be}{\beta}
\newcommand{\Om}{\Omega}
\newcommand{\na}{\nabla}

\newcommand{\cC}{\mathcal{C}}
\newcommand{\cM}{\mathcal{M}}
\newcommand{\nr}{\Vert}
\newcommand{\De}{\Delta}
\newcommand{\cX}{\mathcal{X}}
\newcommand{\cP}{\mathcal{P}}
\newcommand{\om}{\omega}
\newcommand{\si}{\sigma}
\newcommand{\te}{\theta}
\newcommand{\Ga}{\Gamma}

\title[Stability for the Soap Bubble Theorem]{On the stability for \\ Alexandrov's Soap Bubble theorem}

\author{Rolando Magnanini} 
\address{Dipartimento di Matematica ed Informatica ``U.~Dini'',
Universit\` a di Firenze, viale Morgagni 67/A, 50134 Firenze, Italy.}
    %\curraddr{...}
    \email{magnanin@math.unifi.it}
    \urladdr{http://web.math.unifi.it/users/magnanin}

\author{Giorgio Poggesi}
\address{Dipartimento di Matematica ed Informatica ``U.~Dini'',
Universit\` a di Firenze, viale Morgagni 67/A, 50134 Firenze, Italy.}
    %\curraddr{...}
    \email{giorgio.poggesi@unifi.it}

\dedicatory{To Prof. Shigeru Sakaguchi on the occasion of his $60^{th}$ birthday}

\begin{abstract}
Alexandrov's Soap Bubble theorem dates back to $1958$ and states that a compact embedded hypersurface in $\mathbb{R}^N$ with constant mean curvature must be a sphere. For its proof, A.D. Alexandrov invented his reflection priciple. In $1982$, R. Reilly gave an alternative proof, based on integral identities and inequalities, connected with the torsional rigidity of a bar. 
\par
In this article we study the stability of the spherical symmetry: the question is how much a hypersurface is near to a sphere, when its mean curvature is near to a constant in some norm.
\par
We present a stability estimate that states that a compact hypersurface $\Ga\subset\RR^N$ can be contained in a spherical annulus whose interior and exterior radii, say $\rho_i$ and $\rho_e$, satisfy the inequality 
$$
\rho_e - \rho_i \le C \nr H - H_0 \nr^{\tau_N}_{L^1 (\Ga)},
$$
where $\tau_N=1/2$ if $N=2, 3$, and $\tau_N=1/(N+2)$ if $N\ge 4$. Here, $H$ is the mean curvature of $\Ga$, $H_0$ is some reference constant and $C$ is a constant that depends on some geometrical and spectral parameters associated with $\Ga$.  This estimate improves previous results in the literature under various aspects.
\par
We also present similar estimates for some related overdetermined problems.
\end{abstract}


\maketitle

\raggedbottom

\section{Introduction}
Alexandrov's {\it Soap Bubble Theorem} states that a compact hypersurface, embedded in $\mathbb{R}^N$,  that has constant mean curvature $H$ must be a sphere.
To prove that result, A. D. Alexandrov introduced his {\it reflection principle} (see \cite{Al1},\cite{Al2}), later adapted and refined by J. Serrin into the {\it method of moving planes}, that has turned out to be effective to prove radial symmetry of the solutions of certain overdetermined problems in potential theory (see \cite{Se}).
\par
We now know that the fact that essentially the same method works successfully for both problems is not accidental. To see that, we recall that, in its simplest formulation, the result obtained by Serrin states that the overdetermined boundary value problem 
\begin{eqnarray}
\label{serrin1}
&\De u=N \ \mbox{ in } \ \Om, \quad u=0 \ \mbox{ on } \ \Ga, \\
\label{serrin2}
&u_\nu=R \ \mbox{ on } \ \Ga, 
\end{eqnarray}
admits a solution 
for some positive constant $R$ if and only if $\Om$ is a ball of radius $R$ and $u(x)=(|x|^2-R^2)/2$. Here, $\Om$ denotes a bounded domain in $\RR^N$, $N\ge 2$, with sufficiently smooth boundary $\Ga$ and $u_\nu$ is the outward normal derivative of $u$ on $\Ga$. 
\par 
The connection between \eqref{serrin1}-\eqref{serrin2} and the Soap Bubble problem is hinted by the simple differential identity 
$$
\De u=|\na u|\,\dv\frac{\na u}{| \na u|}+\frac{\lan \na^2u\,\na u, \na u\ran}{|\na u|^2};
$$
here, $\na u$ and $\na^2u$ are the gradient and the hessian matrix of $u$, as standard. If we agree
to still denote by $\nu$ the vector field $\na u/|\na u|$ (that on $\Ga$ coincides with the outward unit normal), the above identity and \eqref{serrin1} inform us that
\begin{equation}
\label{reilly}
u_{\nu\nu}+(N-1)\,H\,u_\nu=N,
\end{equation}
on every {\it non-critical} level surface of $u$, and hence on $\Ga$, since a well known formula states that
the mean curvature $H$ of a regular level surface of $u$ equals
$$
\frac1{N-1}\,\dv\frac{\na u}{|\na u|}.
$$
Based on \eqref{reilly}, R.~C.~Reilly  gave in \cite{Re} an alternative proof of Alexandrov's theorem, that hinges on an assortment of integral inequalities and identities and culminates in showing that the inequality 
\begin{equation}
\label{newton}
(\De u)^2\le N\,|\na^2 u|^2,
\end{equation}
that is a simple consequence of Cauchy-Schwarz inequality, holds pointwise with the equality sign on the whole $\Om$. (In our notations, $|\na^2 u|^2$ is the sum of the squares of the entries of $\na^2 u$.) In fact, equality takes place in \eqref{newton} if and only if $u$ is a quadratic polynomial $q$ of the form
\begin{equation}
\label{quadratic}
q(x)=\frac12\, (|x-z|^2-a),
\end{equation}
for some choice of $z\in\RR^N$ and $a\in\RR$, since \eqref{serrin1} is in force. This fact clearly implies that $\Om$ must be a ball.
\par
The aim of this paper is to study the stability issue for the Soap Bubble Theorem. The question is to ascertain under which conditions the smallness (in some norm) of the deviation of $H$ from being a constant implies the closeness of $\Om$ to a ball. The key ingredient to accomplish that goal is the
following integral identity for the solution of \eqref{serrin1}:
\begin{equation}
\label{fundamental}
\frac1{N-1}\int_{\Om} \left\{ |\na ^2 u|^2-\frac{(\De u)^2}{N}\right\}\,dx = N |\Om| -\int_{\Ga}H\, (u_\nu)^2\,dS_x,
\end{equation}
(see Theorem \ref{th:fundamental-identity} below). We shall refer to the first integrand in \eqref{fundamental} as the {\it Cauchy-Schwarz deficit} for $\na^2 u$.
\par
If $H$ is constant on $\Ga$, from  {\it Minkowski's identity},
\begin{equation}
\label{minkowski}
\int_\Ga H(x)\,\lan x-p,\nu(x)\ran\,dS_x=|\Ga|, \quad p\in\RR^N,
\end{equation}
we find that $H\equiv |\Ga|/N |\Om|$ and hence
the H\"older inequality 
\begin{equation}
\label{holder}
\left(\int_\Ga u_\nu\,dS_x\right)^2\le |\Ga|\,\int_\Ga (u_\nu)^2 dS_x
\end{equation}
yield the non-positivity of the right-hand side of \eqref{fundamental}, that gives the equality sign in
\eqref{newton}, as desired. 
\par
If $H$ is not constant, we can take the mean curvature of a ball as a reference under the form 
$$
H_0=\frac{|\Ga|}{N |\Om|};
$$ 
by applying \eqref{holder} as before, from \eqref{fundamental} we obtain that
\begin{equation}
\label{fundamental-stability}
\frac1{N-1}\int_{\Om} \left\{ |\na ^2 u|^2-\frac{(\De u)^2}{N}\right\}\,dx \le\int_{\Ga}(H_0-H)\, (u_\nu)^2\,dS_x.
\end{equation}
\par
It is interesting to note that \eqref{fundamental-stability} implies the spherical symmetry of $u$ (or $\Om$)
if its right-hand side is non-positive, {\it with no need} to use \eqref{minkowski}, and this certainly holds if $H\ge H_0$. (Of course, if $H$ equals some constant on $\Ga$, then \eqref{minkowski} implies that $H\equiv H_0$ and hence $H\ge H_0$, too.) 
\par
Inequality \eqref{fundamental-stability} can also be rearranged as
\begin{multline*}
\label{fundamental-stability2} 
\frac1{N-1}\int_{\Om} \left\{ |\na ^2 u|^2-\frac{(\De u)^2}{N}\right\}\,dx+ \\ \int_{\Ga}(H_0-H)^-\, (u_\nu)^2\,dS_x \le 
\int_{\Ga}(H_0-H)^+\, (u_\nu)^2\,dS_x
\end{multline*}
(here, we use the positive and negative part functions $(t)^+=\max(t,0)$ and $(t)^-=\max(-t,0)$). That inequality tells us that, if we have an {\it a priori} bound $M$ for
$u_\nu$ on $\Ga$, then its left-hand side is small if the integral
\begin{equation*}
\int_{\Ga}(H_0-H)^+\,dS_x
\end{equation*}
is also small.  In particular, if $H$ is not too much smaller than $H_0$, then it cannot be too much larger than $H_0$ and the Cauchy-Schwarz deficit cannot be too large.  Thus, to achieve our aim, it remains to quantitavely transform this smallness into closeness
of $\Om$ to a ball.
\par
In Theorem \ref{thm:SBT-stability}, we shall prove that, for some point $z\in\Om$ the radius of
the largest ball centered at $z$ and contained in $\Om$
and that of the smallest concentric ball that contains $\Om$, that is
\begin{equation}
\label{def-rhos}
\rho_i=\min_{x\in\Ga}|x-z| \ \mbox{ and } \ \rho_e=\max_{x\in\Ga}|x-z|,
\end{equation}
satisfy the following stability estimate:
\begin{equation}
\label{the-estimate}
\rho_e-\rho_i\le C\,\left\{\int_{\Ga}(H_0-H)^+\,dS_x\right\}^{\tau_N},
\end{equation}
where $\tau_N=1/2$ for $N=2, 3$ and $\tau_N=1/(N+2)$ for $N\ge 4$.
Here, the constant $C$ depends on $N$ and some geomerical and spectral parameters associated with $\Om$ (see Theorem \ref{thm:SBT-stability} for details). 
\par
Inequality \eqref{the-estimate} improves similar estimates given in \cite{CV} and \cite{CM}, under various aspects. In fact,
it replaces the uniform measure $\nr H-H_0\nr_{\infty,\Ga}$ of the deviation from $H_0$, considered in \cite{CV} and \cite{CM}, by a weaker $L^1$-type norm; 
it is thus extended to a larger class of hypersurfaces, being not restricted (as it is in \cite{CM}) to those with positive mean curvature.  Also, it improves the exponent $\tau_N$ obtained in \cite{CM} --- even in this weaker setting and for all $N\ge 2$ --- to the extent that it obtains, for the cases $N=2, 3$, what seems to be the optimal exponent attainable {\it  with this approach.} Finally, even if it does not improve the exponent $\tau_N=1$, obtained in \cite{CV} by means of an adaptation of the reflection principle, it favours the computability of the constant $C$, as shown in \cite{CM} and differently from \cite{CV}.
\par
To prove \eqref{the-estimate}, we consider the function  $h=q-u$; $h$ is harmonic in $\Om$, $h=q$ on $\Ga$,  and we can deduce from
\eqref{fundamental-stability} that
\begin{equation*}
\label{fundamental-harmonic}
\frac1{N-1}\int_{\Om} |\na ^2 h|^2\,dx \le\int_{\Ga}(H_0-H)^+\, (u_\nu)^2\,dS_x.
\end{equation*}
Notice that this inequality holds regardless of the choice of the parameters $z\in\RR^N$ and $a\in\RR$ in \eqref{quadratic}. Thus, to ensure that $z$ is in $\Om$, we choose it as a minimum (or any critical) point of $u$; then,
since $\na h(z)=0$, we show that the oscillation of $h$ on $\Ga$,
$$
\max_\Ga h-\min_\Ga h,
$$
can be  bounded in terms of a power of the quantity
$$
\int_{\Om} |\na ^2 h|^2\,dx.
$$
Since $h$ attains its extrema when $q$ does, then
\begin{equation}
\label{oscillation}
\max_\Ga h-\min_\Ga h=\frac12 (\rho_e^2-\rho_i^2),
\end{equation}
and hence an a priori bound for $u_\nu$ on $\Ga$ and the observation that $\rho_i+\rho_e$ can be bounded from below by the volume of $\Om$ give the desired estimate. 
\par
In Section \ref{sec:SBT}, we shall collect all the relevant identities on which our result is based. To make the presentation self-contained, we will also include a version of Reilly's proof of the Soap Bubble Theorem.  We will also discuss versions of those identities that give radial symmetry for some overdetermined problems associated with \eqref{serrin1}. In particular, we will present a new proof of Serrin's symmetry result that, however, only works if $\Om$ is strictly star-shaped with respect to some origin.
\par
Section \ref{sec:estimates} contains the estimates on harmonic functions and the torsional creep function that are instrumental to derive \eqref{the-estimate}. The key result is Lemma \ref{thm:W22-stability}, in which we are able to bound the difference $\rho_e-\rho_i$ in terms of 
$\nr \na^2 h\nr_{2,\Om}$. Theorem \ref{thm:boundary-gradient} provides a simple bound for the gradient of $u$ on $\Ga$ in terms of the diameter of $\Om$ and the radius of the exterior uniform touching ball. This bound is important, since it allows to treat the general case of $C^{2,\al}$-smooth hypersurfaces, and is obtained by elementary arguments and seems to be new,  generalizing the classical work of Payne and Philippin \cite{PP}, that concerned the case of strictly mean convex domains, in which $H$ is positive at each point in $\Ga$.
\par
Finally, in Section \ref{sec:stability}, we assemble the identities and inequalities proved in the previous sections and establish our stability results. As a corollary of our main inequality contained in Theorem \ref{thm:SBT-stability}, we obtain an estimate of closeness to an aggregate of balls, in the spirit of \cite{CM}. With more or less the same techniques employed for Theorem \ref{thm:SBT-stability}, we also present stability bounds for some of the overdetermined problems considered in Section \ref{sec:SBT}.

\section{Alexandrov's Soap Bubble Theorem}
\label{sec:SBT}

In this section, we will review the details of Reilly's proof, with some modifications, that will enable us to 
derive our stability results. The proof we present is based on the identity \eqref{fundamental}. We shall also show how to use \eqref{fundamental} to obtain other symmetry results, old and new.
\par
The identity \eqref{fundamental} is a consequence of the differential identity for the solution $u$
of \eqref{serrin1},
\begin{equation}
\label{differential-identity}
|\na^2 u|^2-\frac{(\De u)^2}{N}=\De P,
\end{equation}
that associates the Cauchy-Schwarz deficit with the {\it P-function}
\begin{equation}
\label{P-function}
P = \frac{1}{2}\,|\nabla u|^2 - u,
\end{equation}
and is easily obtained by direct computation. Notice that \eqref{differential-identity} also implies
that $P$ is subharmonic, since the left-hand side is non-negative by Cauchy-Schwarz inequality.
\par
In the next theorem, for the sake of completeness, we give the proof of \eqref{fundamental}, that can also be found in \cite{Re}.

\begin{thm}[Fundamental Identity]
\label{th:fundamental-identity}
Let $\Omega \subset \mathbb R^N$ be a bounded domain with boundary $\Ga$ of class $C^{2, \alpha}$ and let $H$ be the mean curvature of $\Ga$.
\par
If $u$ is the solution of  \eqref{serrin1},
then \eqref{fundamental} holds:
\begin{equation*}
\frac1{N-1}\int_{\Om} \left\{ |\na ^2 u|^2-\frac{(\De u)^2}{N}\right\}\,dx = N |\Om| -\int_{\Ga}H\, (u_\nu)^2\,dS_x.
\end{equation*}
\end{thm}

\begin{proof}
Let $P$ be given by \eqref{P-function}. By the divergence theorem we can write:
\begin{equation}
\label{div-p-function}
\int_{\Om} \De P \,dx = \int_{\Ga} P_\nu \, dS_x.
\end{equation}
To compute $P_\nu$, we observe that $\na u$ is
parallel to $\nu$ on $\Ga$, that is $\na u=(u_\nu)\,\nu$ on $\Ga$. Thus,
$$
P_\nu=\lan D^2 u\, \na u, \nu\ran-u_\nu=u_\nu \lan (D^2 u)\,\nu,\nu\ran-u_\nu=u_{\nu\nu}\, u_\nu-u_\nu.
$$
By Reilly's identity \eqref{reilly}, we know that
$$
u_{\nu\nu}\, u_\nu+(N-1)\,H\, (u_\nu)^2=N\,u_\nu,
$$
and hence
$$
P_\nu=(N-1)\,u_\nu-(N-1)\,H\, (u_\nu)^2
$$
on $\Ga$.
\par
Therefore, \eqref{fundamental} follows from this identity, \eqref{differential-identity}, \eqref{div-p-function} and the formula
\begin{equation}
\label{volume}
\int_\Ga u_\nu\,dS_x=N\,|\Om|,
\end{equation}
that is an easy consequence of the divergence theorem.
\end{proof}

The fundamental identity \eqref{fundamental} can be re-arranged at least into two ways to yield the Soap Bubble Theorem. The former follows the lines of Reilly's proof. The latter gives Alexandrov's theorem via the Heintze-Karcher's inequality \eqref{heintze-karcher} below and we will present it at the end of this section.

\begin{thm}[Soap Bubble Theorem] 
\label{th:SBT}
Let $\Ga\subset\RR^N$ be a surface of class  $C^{2, \alpha}$, which is the boundary of a bounded domain $\Om\subset\RR^N$, and let $u$ be the solution of \eqref{serrin1}.  
Let two positive constants be defined by
\begin{equation}
\label{R and H_0}
R=\frac{N |\Om|}{|\Ga|} \ \mbox{ and } \ H_0=\frac1{R}=\frac{|\Ga|}{N |\Om|}.
\end{equation}
\par
Then, the following identity holds:
\begin{multline}
\label{H-fundamental}
\frac1{N-1}\int_{\Om} \left\{ |\na ^2 u|^2-\frac{(\De u)^2}{N}\right\}dx+
\frac1{R}\,\int_\Ga (u_\nu-R)^2 dS_x = \\
\int_{\Ga}(H_0-H)\, (u_\nu)^2 dS_x.
\end{multline}
\par
Therefore, if the mean curvature $H$ of $\Ga$ satisfies the inequality $H\ge H_0$ on $\Ga$,
then $\Ga$ must be a sphere (and hence $\Omega$ is a ball) of radius $R$. 
\par
In particular, the same conclusion holds if $H$ equals some constant on $\Ga$.

\end{thm}

\begin{proof}
Since, by \eqref{volume}, we have that
$$
\frac1{R}\,\int_\Ga u_\nu^2\,dS_x=
\frac1{R}\,\int_\Ga (u_\nu-R)^2\,dS_x+N |\Om|,
$$
then
\begin{multline*}
\int_\Ga H\,(u_\nu)^2\,dS_x=H_0\,\int_\Ga (u_\nu)^2\,dS_x+
\int_\Ga \left( H-H_0\right)\,(u_\nu)^2\,dS_x= \\
\frac1{R}\,\int_\Ga (u_\nu-R)^2\,dS_x+N |\Om|+
\int_\Ga ( H-H_0)\,(u_\nu)^2\,dS_x.
\end{multline*}
Thus, \eqref{H-fundamental} follows from this identity and \eqref{fundamental} at once.
\par
If $H\ge H_0$ on $\Ga$, then the right-hand side in \eqref{H-fundamental} is non-positive and hence both summands at the left-hand side must 
be zero, being non-negative. (Note in passing that this fact implies that the second summand is zero and hence $u_\nu\equiv R$ on $\Ga$, that is $u$ satisfies \eqref{serrin1}-\eqref{serrin2}.)
\par
The fact that also the first summand is zero gives that the Cauchy-Schwarz deficit for the hessian matrix $\na^2 u$ must be identically zero and, being $\De u=N$, that occurs if and only if $\na^2 u$ equals the identity matrix $I$.  Thus, $u$ must be a quadratic polynomial $q$, as in \eqref{quadratic},
for some $z\in\RR^N$ and $a\in\RR$.
\par
Since $u=0$ on $\Ga$, then $|x-z|^2=a$ for $x\in\Ga$, that is $a$ must be positive and
$$
\sqrt{a}\,|\Ga|=\int_\Ga |x-z|\,dS_x=\int_\Ga (x-z)\cdot\nu(x)\,dS_x=N\,|\Om|.
$$
\par
In conclusion, $\Ga$ must be a sphere centered at $z$ with radius $R$.
\par
If $H$ equals some constant, instead, then \eqref{minkowski} tells us that the constant must equal $H_0$, and hence we can apply the previous argument.
\end{proof}

\begin{rem}
{\rm
(i) As pointed out in the previous proof, {\it before showing that $\Om$ is a ball}, we have also proved that, if $H$ is constant, then  
$u$ satisfies \eqref{serrin1}-\eqref{serrin2}. It would be interesting to show that also the converse is true. That would show that the two problems are equivalent.
\par
(ii) We observe that the assumption that $H\ge H_0$ on $\Ga$ implies that $H\equiv H_0$, anyway, if $\Om$ is strictly star-shaped with respect to some origin $p$. In fact, by Minkowski's identity \eqref{minkowski}, we obtain that
$$
0\le\int_\Ga [H(x)-H_0] \lan(x-p), \nu(x)\ran\,dS_x=|\Ga|-H_0 \int_\Ga \lan(x-p), \nu(x)\ran\,dS_x=0,
$$
and we know that $\lan(x-p), \nu(x)\ran>0$ for $x\in\Ga$.
}
\end{rem}

\par
Before presenting the proof of Alexandrov's theorem based on the Heintze-Karcher's inequality, we  prove two symmetry results for overdetermined problems (one of which 
is Serrin's result under some restriction), that have their own interest.
 
\begin{thm}[Two overdetermined problems]
\label{th:torsion}
Let $u\in C^{2,\al}(\ol{\Om})$ be the solution of \eqref{serrin1}.
\par
Then, $\Om$ is a ball if and only if $u$ satisfies one of the following conditions:
\begin{enumerate}[(i)]

\item
$u_\nu(x)=1/H(x)$ for every $x\in\Ga$;

\item
\eqref{serrin2} holds and $\lan(x-p), \nu(x)\ran>0$ for every $x\in\Ga$ and some $p\in\Om$.
\end{enumerate}
\end{thm}

\begin{proof}
It is clear that, if $\Om$ is a ball, then (i) and (ii) hold. Conversely, we shall check that the right-hand side of \eqref{fundamental} is zero when one of the items (i) or (ii) occurs.
\par
(i) Notice that our assumption implies that $H$ must be positive, since $u_{\nu}$ is positive and finite. Since \eqref{volume} holds, then
\eqref{fundamental} can be written as
\begin{equation}
\label{fundamental-identity2}
\frac1{N-1}\int_{\Om} \left\{ |\na ^2 u|^2-\frac{(\De u)^2}{N}\right\}\,dx = \int_{\Ga}(1-H u_\nu)\, u_\nu\,dS_x,
\end{equation}
and the conclusion follows at once.
\par
(ii) 
Let $u_\nu$ be constant on $\Ga$; by \eqref{volume} we know that that constant equals the value $R$ given in \eqref{R and H_0}. Also, notice that $1-H u_\nu\ge 0$ on $\Ga$. In fact, the function $P$ in \eqref{P-function}
is subharmonic in $\Om$, since $\De P\ge 0$ by \eqref{differential-identity}. Thus, it attains its maximum on $\Ga$, where it is constant.
We thus have that
$$
0\le P_\nu=u_\nu u_{\nu \nu}-u_\nu=(N-1)\,(1-H u_\nu)\,u_\nu \ \mbox{ on } \ \Ga.
$$
Now,
\begin{multline*}
0\le\int_\Ga [1-H(x) u_\nu(x)]\,\lan (x-p),\nu(x)\ran\,dS_x=\\
\int_\Ga [1-H(x) R]\,\lan (x-p),\nu(x)\ran\,dS_x=0,
\end{multline*}
by \eqref{minkowski}. Thus, $1-H u_\nu\equiv 0$ on $\Ga$ and hence (i) applies. 
\end{proof}

\begin{rem}
{\rm
The proof of (ii) seems to be new. Even if it is restricted to the case of strictly star-shaped domains, it might be
used to obtain better stability estimates for Serrin's symmetry result. 
\par
We recall that, by following the tracks of Weinberger's proof (\cite{We}) and its modification due to Payne and Schaefer (\cite{PS}), one
can write the identity
\begin{equation}
\label{wps}
\int_{\Om} (-u)\,\left\{ |\na ^2 u|^2- \frac{ (\De u)^2}{N} \right\}\,dx=
\frac{1}{2}\,\int_\Ga (u_\nu^2-R^2)\,(u_\nu-x\cdot\nu)\,dS_x,
\end{equation}
that gives at once spherical symmetry if $u_\nu$ is constant on $\Ga$, without major restrictions on $\Om$ other than on the regularity of $\Ga$. The presence of the factor $-u$ at the left-hand side, however, may cause additional difficulties in the study of the stability issue.
}
\end{rem}

\medskip

We conclude this section by showing that \eqref{fundamental} can be rearranged into an identity that implies Heintze-Karcher's inequality (see \cite{HK}). This proof is slightly different from that of A. Ros in \cite{Ro} and relates the equality case for Heintze-Karcher's inequality to the overdetermined problem considered in (i) of Theorem \ref{th:torsion}.

\begin{thm}[SBT and Heintze-Karcher inequality]
\label{th:heintze-karcher}
Let $\Ga\subset\RR^N$ be a surface of class  $C^{2, \alpha}$, which is the boundary of a bounded domain $\Om\subset\RR^N$ and let $u\in C^{2,\al}(\ol{\Om})$ be the solution of \eqref{serrin1}. 
\par
If $\Ga$ is strictly mean-convex, then we have the following identity:
\begin{equation}
\label{heintze-karcher-identity}
\frac1{N-1}\,\int_{\Om} \left\{ |\na ^2 u|^2-\frac{(\De u)^2}{N}\right\}\,dx +\int_\Ga\frac{(1-H\,u_\nu)^2}{H}\,dS_x=
\int_\Ga\frac{dS_x}{H}-N |\Om|.
\end{equation}
\par
In particular, the Heintze-Karcher's inequality
\begin{equation}
\label{heintze-karcher}
\int_\Ga \frac{dS_x}{H}\ge N |\Om|
\end{equation}
holds and the sign of equality 
is attained in \eqref{heintze-karcher} if and only if $\Om$ is a ball.
\par
Thus, if $H$ is constant on $\Ga$, then $\Ga$ is a sphere. 
\end{thm}

\begin{proof}
By integrating on $\Ga$ the identity
$$
\frac{(1-H\,u_\nu)^2}{H}=-(1-H\,u_\nu) u_\nu+\frac1{H}-u_\nu,
$$
summing the result up to \eqref{fundamental-identity2} and taking into account \eqref{volume},
we get \eqref{heintze-karcher-identity}.
\par
Both summands at the left-hand side of \eqref{heintze-karcher-identity}  are non-negative and hence \eqref{heintze-karcher} follows. If the right-hand side is zero, those summands must be zero. The vanishing of the first summand implies that $\Om$ is a ball, as already noticed. Note in passing that the vanishing of the second summand gives that $u_\nu=1/H$ on $\Ga$, which also implies radial symmetry, by Theorem \ref{th:torsion}.  
\par
Finally, if $H$ equals some constant on $\Ga$, we know that such a constant must have the value $H_0$ in \eqref{R and H_0}, that implies that the right-hand side of  \eqref{heintze-karcher-identity} is null and hence, once again, $\Om$ must be a ball.
\end{proof}

\section{Some estimates for harmonic functions} 
\label{sec:estimates}

We begin by setting some relevant notations. By $\Ga\subset\RR^N$, $N\ge 2$, we shall always denote a hypersurface of class $C^{2,\al}$, $0<\al<1$, that is the boundary of a bounded domain $\Om$. 
By $|\Om|$ and $|\Ga|$, we will denote indifferently the $N$-dimensional Lebesgue measure of $\Om$
and the surface measure of $\Ga$. The {\it diameter} of $\Om$ will be indicated by $d_\Om$. 
\par
Moreover, since $\Ga$ is bounded and of class $C^{2,\al}$, it has the properties of the {\it uniform interior and exterior sphere condition}, whose respective radii will be designated by $r_i$ and $r_e$;
namely, there exists $r_e > 0$ (resp. $r_i>0$) such that for each $p \in \Ga$ there exists a ball $B \subset \RR^N \setminus \ol{\Om}$ (resp. $B\subset\Om$) of radius $r_e$ (resp. $r_i$) such that
$\ol{B} \cap\Ga= \{ p \}$.
\par
The assumed regularity of $\Ga$ ensures that the unique solution of \eqref{serrin1} is of class $C^{2,\al}$. Thus, we can define
\begin{equation}
\label{bound-gradient}
M=\max_{\ol{\Om}} |\na u|=\max_{\Ga} u_\nu.
\end{equation}
 \par
We finally recall that, for a point $z\in\Om$, $\rho_i$ and $\rho_e$ denote 
the radius of the largest ball centered at $z$ and contained in $\Om$
and that of the smallest ball that contains $\Om$ with the same center, as defined in \eqref{def-rhos}.
\par
As already mentioned in the Introduction, the first summand in \eqref{H-fundamental} or in \eqref{heintze-karcher-identity} can be suitably re-written, in terms of the harmonic function $h=q-u$, as
\begin{equation}
\label{L2-norm-hessian}
\frac1{N-1}\int_{\Om} \left\{ |\na ^2 u|^2-\frac{(\De u)^2}{N}\right\}dx = \frac1{N-1}\,\int_\Om |\na^2 h|^2 dx,
\end{equation}
where $q$ is any quadratic polynomial of the form \eqref{quadratic}.
Also, if we choose the center $z$ of the paraboloid \eqref{quadratic} in $\Om$, 
we have \eqref{oscillation}, that is 
$$
\max_{\Ga} h-\min_{\Ga} h=\max_{\Ga} q-\min_{\Ga} q=\frac12\,\left(\rho_e^2-\rho_i^2\right).
$$
\par
A stability estimate for the spherical symmetry of $\Ga$ will then be obtained, via identity 
\eqref{H-fundamental} or \eqref{heintze-karcher-identity}, if we associate the oscillation of $h$ on $\Ga$
with \eqref{L2-norm-hessian}. 

\medskip

To realize this agenda, we start by proving some Poincar\'e-type inequalities for harmonic functions.

\begin{lem}[Poincar\'e-type inequalities]
\label{lem:two-inequalities}
There exist two positive constants, $\ol{\mu}(\Om)$ and $\mu_0(\Om)$, such that
\begin{equation}
\label{harmonic-poincare}
\int_\Om v^2 dx\le \ol{\mu}(\Om)^{-1} \int_\Om |\na v|^2 dx,
\end{equation}
or
\begin{equation}
\label{harmonic-quasi-poincare}
\int_\Om v^2 dx\le \mu_0(\Om)^{-1} \int_\Om |\na v|^2 dx,
\end{equation}
for every function $v\in W^{1,2}(\Om)$ which is harmonic in $\Om$ and such that
\begin{equation}
\label{normalization1}
\int_\Om v\,dx=0 
\end{equation}
or, respectively,
\begin{equation}
\label{normalization2}
v(x_0)=0,
\end{equation}
where $x_0$ is a given point in $\Om$. 
\end{lem}

\begin{proof}
We define
\begin{equation}
\label{inf-harmonic-poincare}
\ol{\mu}(\Om)=\inf\left\{\int_\Om |\na v|^2\,dx: \int_\Om v^2 dx=1,\, \De v=0 \mbox{ in } \Om, \int_\Om v\,dx=0\right\}
\end{equation}
and
\begin{equation}
\label{inf-harmonic-poincare2}
\mu_0(\Om)=\inf\left\{\int_\Om |\na v|^2\,dx: \int_\Om v^2 dx=1, \De v=0 \mbox{ in } \Om, v(x_0)=0\right\}.
\end{equation}
If we prove that the two infima are attained, we obtain that they are positive and hence \eqref{harmonic-poincare} and \eqref{harmonic-quasi-poincare} hold.
\par
We know that we can find a minimizing sequence $\{ v_n\}_{n\in\NN}$ of \eqref{inf-harmonic-poincare}
or \eqref{inf-harmonic-poincare2}, that converges in $L^2(\Om)$ and weakly in $W^{1,2}(\Om)$ to a function $v$ in $W^{1,2}(\Om)$. Also, by the mean value property for harmonic functions, this sequence converges uniformly on
compact subsets of $\Om$, that implies that $v$ is harmonic in $\Om$. Thus, 
we easily infer that 
$$
\int_\Om v^2 dx=1,
$$
and $\int_\Om v\,dx=0$, if $\{ v_n\}_{n\in\NN}$ is minimizing the problem \eqref{inf-harmonic-poincare}, or $v(x_0)=0$, if it is minimizing \eqref{inf-harmonic-poincare2}.
\par
Finally, we have that
$$
\ol{\mu}(\Om)=\liminf_{n\to\infty}\int_\Om|\na v_n|^2 dx\ge\int_\Om |\na v|^2 dx\ge\ol{\mu}(\Om),
$$
by the weak convergence in $W^{1,2}(\Om)$. This same conclusion holds for problem \eqref{inf-harmonic-poincare2}.
\end{proof}

\begin{rem}
\label{rem:eigenvalues}
{\rm
It is clear that 
$
\ol{\mu}(\Om)>\mu_2(\Om),
$
where $\mu_2(\Om)$ is the {\it second Neumann eigenvalue}. Moreover,
$$
\mu_0(\Om)\le\ol{\mu}(\Om).
$$
In fact, let
$$
v_0=\frac{v-v(x_0)}{1+|\Om|\,v(x_0)^2},
$$
where $v$ is a minimizer for \eqref{inf-harmonic-poincare}; $v_0$ is harmonic in $\Om$, $v_0(x_0)=0$, and $\int_\Om v_0^2\,dx=1$. Therefore,
$$
\mu_0(\Om)\le\int_\Om |\na v_0|^2\,dx=\frac{\int_\Om |\na v|^2\,dx}{[1+|\Om|\, v(x_0)^2]^2}=
\frac{\ol{\mu}(\Om)}{[1+|\Om|\, v(x_0)^2]^2}\le\ol{\mu}(\Om).
$$
}\end{rem}

\bigskip

The crucial result is Theorem \ref{thm:W22-stability}, in which we associate the oscillation of the already defined harmonic function $h=q-u$,
and hence the difference $\rho_e-\rho_i$, with the $L^2$-norm $\nr \na^2 h\nr_{2,\Om}$ of its Hessian matrix. In the following lemma, we start by linking that oscillation with the $L^2$-norm of $h-h(z)$.
To this aim, we define the {\it parallel set} as
$$
\Om_\si=\{ y\in\Om: \dist(y,\Ga)>\si\} \quad \mbox{ for } \quad 0<\si<r_i.
$$
.
\begin{lem}
\label{lem:L2-estimate-oscillation}
Let $\Om\subset\RR^N$, $N\ge 2$, be a bounded domain with boundary of class $C^{2,\al}$.
Set $h=q-u$, where $u$ is the solution of \eqref{serrin1} and $q$ is any quadratic polynomial as in \eqref{quadratic} with $z\in\Om$.
\par
Then, if
\begin{equation}
\label{smallness}
\nr h-h(z)\nr_2< \frac{\sqrt{|B|}}{N\,2^{N+1}}\,M\, r_{i}^{\frac{N+2}{2}} 
\end{equation}
holds, we have that
\begin{equation}
\label{L2-stability}
\rho_e-\rho_i\le a_N\, \frac{ M^{ \frac{N}{N+2} } }{ |\Om|^{\frac{1}{N}} }\, \nr h-h(z)\nr_2^{ 2/(N+2) },
\end{equation}
where

\begin{equation}
\label{provacondefcostaN}
a_N= \frac{ 2^{2+ \frac{N}{N+2}} \, (N+2)}{N^{\frac{N}{N+2}}} \,|B|^{\frac{1}{N} - \frac{1}{N+2}}.
\end{equation}
\end{lem}

\begin{proof}
Let $x_i$ and $x_e$ be points in $\Ga$ that minimize (resp. maximize) $q$ on $\Ga$ and, for 
$$
0<\si<r_i,
$$ define the two points in $y_i, y_e\in\pa\Om_\si$ by
$y_j=x_j-\si\nu(x_j)$, $j=i, e$. 
\par
We have that
$$
h(y_j)-h(x_j)=-\int_0^\si \lan\na h(x_j-t\nu(x_j)),\nu(x_j)\ran\,dt
$$
and hence, by recalling that $\na h(x)=x-z-\na u(x)$ and that $x_j-z$ is parallel to $\nu(x_j)$, we obtain that
$$
h(y_j)-h(x_j)=-|x_j-z|\,\si+\frac12\,\si^2+\int_0^\si \lan\na u(x_j-t\nu(x_j)),\nu(x_j)\ran\,dt.
$$
Thus, the fact that $2\,h(x_j)=\rho_j^2$ and $|x_j-z|=\rho_j$ yields that
\begin{equation*}
\frac12\,\rho_j^2-\rho_j\,\si+\frac12\,\si^2=h(y_j)-\int_0^\si \lan\na u(x_j-t\nu(x_j)),\nu(x_j)\ran\,dt, \ j=i, e,
\end{equation*}
and hence
\begin{equation}
\label{parallel estimate vecchia1}
\frac12\,(\rho_e-\rho_i)(\rho_e+\rho_i-2\si)\le h(y_e)-h(y_i)+2 M \si,
\end{equation}
for every $0<\si< \min \left\lbrace  \frac{ \rho_e+\rho_i }{2} , r_i \right\rbrace $.
\par
Since $h$ is harmonic and $y_j\in \overline{ \Om }_\si $, $j=i, e$, we can use the mean value property for the balls with radius $\si$ centered at  $y_j$ and obtain: 
\begin{multline*}
|h(y_j)-h(z)|\le \frac1{|B|\, \si^N}\,\int_{B_\si(y_j)}|h-h(z)|\,dy\le \\
\frac1{\sqrt{|B|\, \si^N}}\,\left[\int_{B_\si(y_j)}|h-h(z)|^2\,dy\right]^{1/2}\le 
\frac1{\sqrt{|B|\, \si^N}}\,\left[\int_{\Om}|h-h(z)|^2\,dy\right]^{1/2}.
\end{multline*}
by H\"older's inequality. This and inequality \eqref{parallel estimate vecchia1} then yield that
\begin{equation*}
\frac{1}{2} \, (\rho_e+\rho_i-2\si)(\rho_e-\rho_i)\le 2 \, \left[  \frac{\nr h-h(z)\nr_2}{\sqrt{|B|}\,\si^{N/2}}+ M \si \right] ,
\end{equation*}
for every $0<\si< \min \left\lbrace  \frac{\rho_e+\rho_i }{2} , r_i \right\rbrace $.
\par
Now, observe that, for $0<\si<\si_0$ with
$$
\si_0=\frac14\,\left(\frac{|\Om|}{|B|}\right)^{1/N},
$$
then $\rho_e+\rho_i-2\si>2\si_0$, and hence
\begin{equation}
\label{parallel-estimate}
\rho_e-\rho_i\le \frac{2}{\si_0}\left[\frac{\nr h-h(z)\nr_2}{\sqrt{|B|}\,\si^{N/2}}+M \si\right]
\ \mbox{ for every } \ 0<\si< \min \left\lbrace  \si_0 , r_i \right\rbrace .
\end{equation}
Therefore, by minimizing the right-hand side of \eqref{parallel-estimate}, we can conveniently choose 
$$
\si=\left(\frac{N\,\nr h-h(z)\nr_2 }{ 2 \, |B|^{1/2}\,M}\right)^{2/(N+2)} 
$$
in \eqref{parallel-estimate} and obtain \eqref{L2-stability}, if $\si < r_i/4<\min\{  \si_0 , r_i\}$;  \eqref{smallness} will then follow.  
\end{proof}

\medskip

To simplify formulas, in the remainder of this section and in Section \ref{sec:stability}, we shall always denote the constants only depending on the dimension by $k_N$ and $\al_N$. Their computation will be clear from the relevant proofs.

\medskip

A way to conveniently choose $z$ inside $\Om$ is to let $z$ be any (local) minimum point of $u$ 
in $\ol{\Om}$: we are thus sure that $z\in\Om$ and, also, we obtain that $\na h(z)=0$.
This remark and Lemmas \ref{lem:two-inequalities} and \ref{lem:L2-estimate-oscillation} give the following result.

\begin{thm}
\label{thm:W22-stability}
Let $z$ be any (local) minimum point of the solution $u$ of \eqref{serrin1} in $\ol{\Om}$. Set 
$h=q-u$, where $q$ is given by \eqref{quadratic}.
\par
If $N=2$ or $3$, then 
\begin{equation}
\label{Lipschitz-stability}
\rho_e-\rho_i\le C\,\nr \na^2 h\nr_{2,\Om},
\end{equation}
where
$$
C= k_N \, c \, \frac{ d_\Om^\ga}{|\Om|^{ \frac{1}{N}}} \, \frac{1+\mu_0(\Om)}{\mu_0(\Om)},
$$
$\ga$ is any number in $(0,1)$ for $N=2$, $\ga=1/2$ for $N=3$, and $c$ is the Sobolev
immersion constant of $C^{0,\ga}(\ol{\Om})$ in $W^{2,2}(\Om)$. 
\par
If else $N\ge 4$, then
\begin{equation}
\label{W22-stability}
\rho_e-\rho_i\le C\, \nr \na^2 h\nr_{2,\Om}^{2/(N+2)} 
\end{equation}
for
\begin{equation}
\label{W22-smallness}
\nr \na^2 h\nr_{2.\Om}<\ve,
\end{equation}
where
$$
C=k_N\, \frac{ M^{ \frac{N}{N+2} } }{ \mu_0(\Om)^{ \frac{2}{N+2} } \, |\Om|^{ \frac{1}{N} } }   \quad
\mbox{ and }
\quad
\ve= \al_N \,M \, \mu_0(\Om) \, r_{i}^{\frac{N+2}{2}} \, .
$$
\end{thm}
\begin{proof}
(i) Let $N=2$ or $3$.
By the Sobolev immersion theorem (see for instance \cite[Theorem 3.12]{Gi} or \cite[Chapter 5]{Ad}), we have that there is a constant $c$ such that, for any $v\in W^{2,2}(\Om)$, we have that
$$
\frac{|v(x)-v(y)|}{|x-y|^\ga}\le c\,\nr v\nr_{W^{2,2}(\Om)} \ \mbox{ for any  $x, y\in\ol{\Om}$ with $x\not=y$},
$$
where $\ga$ is any number in $(0,1)$ for $N=2$ and $\ga=1/2$ for $N=3$. 
\par
We now set $v=h-h(z)$, $x_0=z$, and apply \eqref{harmonic-quasi-poincare} twice: to $v$ and
to each first derivative of $v$ (since $\na v(z)=\na h(z)=0$). We obtain that
$$
\nr h-h(z)\nr_{W^{2,2}(\Om)}\le \sqrt{1+\mu_0(\Om)^{-1}+\mu_0(\Om)^{-2}}\,\nr \na^2 h\nr_{2,\Om}.
$$
Since $h-h(z)$ is harmonic, it attains its extrema on $\Ga$ and hence we have that
$$
\frac12\,(\rho_e^2-\rho_i^2)=\max_\Ga h-\min_\Ga h\le c\,d_\Om^\ga\, \sqrt{1+\mu_0(\Om)^{-1}+\mu_0(\Om)^{-2}}\,\nr \na^2 h\nr_{2,\Om}.
$$
 Thus, \eqref{Lipschitz-stability} follows by observing that $\rho_e+\rho_i\ge \rho_e\ge (|\Om|/|B|)^{1/N}$ and that the square root can be bounded by $1+\mu_0(\Om)^{-1}$.
\par
(ii) Let $N\ge 4$. By the same choice of $v$ and $x_0$ as in (i), we obtain that
$$
\nr h-h(z)\nr_{2,\Om}\le \mu_0(\Om)^{-1}\nr\na^2 h\nr_{2,\Om}.
$$
The conclusion then follows from Lemma \ref{lem:L2-estimate-oscillation}.
\end{proof}

\begin{rem}
{\rm
We recall that if $\Om$ has the strong local Lipschitz property (for the definition see \cite[Section 4.5]{Ad}), the immersion constant $c$ depends only on $N$ and the two Lipschitz parameters of the definition (see \cite[Chapter 5]{Ad}).
In our case $\Om$ is of class $C^{2,\al}$, hence obviously it has the strong local Lipschitz property and the two Lipschitz parameters can be easily estimated in terms of $\min \lbrace r_i,r_e \rbrace $.
}
\end{rem}

If we use \eqref{harmonic-poincare} instead of \eqref{harmonic-quasi-poincare},  we obtain a similar  result, but we must suppose that $\Om$ contains its center of mass. 

\begin{thm}
\label{th:W22-stability-neumann}
Let $z$ be the center of mass of $\Om$ and suppose that $z\in\Om$.  
Set $h=q-u$, where $q$ is given by \eqref{quadratic} and
set the constant $a$ in $q$ such that
$$
\int_\Om [h(x)-h(z)] dx=0 .
$$
\par
If $N=2$ or $3$, then \eqref{Lipschitz-stability} holds
with
$$
C= k_N \, c \, \frac{ d_\Om^\ga}{|\Om|^{ \frac{1}{N}}} \, \frac{1+\ol{\mu}(\Om)}{\ol{\mu}(\Om)},
$$
where $\ga$ and $c$ are the constants introduced in Theorem \ref{thm:W22-stability}.
\par
If else $N\ge 4$, then \eqref{W22-stability} holds if \eqref{W22-smallness} is in force, where
$$
C=k_N \, \frac{ M^{ \frac{N}{N+2} } }{ \ol{\mu}(\Om)^{ \frac{2}{N+2} } \, |\Om|^{ \frac{1}{N} } }   \quad
\mbox{ and }
\quad
\ve= \al_N \,M \, \ol{\mu}(\Om) \, r_{i}^{\frac{N+2}{2}} \, .
$$
\end{thm}

\begin{proof}
The proof is similar to that of Theorem \ref{thm:W22-stability}. 
Since $z$ is the center of mass of $\Om$, we have that
\begin{multline*}
\int_\Om \na h(x)\,dx=\int_\Om [x-z-\na u(x)]\,dx=\\
\int_\Om x\,dx-|\Om|\,z-\int_\Ga u(x)\,\nu(x)\,dS_x=0.
\end{multline*}
We can thus apply \eqref{harmonic-poincare} to the first derivatives of $h$ and obtain that
$$
\nr \na h\nr_{2,\Om}\le \ol{\mu}(\Om)^{-1/2} \nr \na^2h\nr_{2,\Om}.
$$
Since we chose the constant $a$ in $q$ such that  
$$
\int_\Om [h(x)-h(z)] dx=0,
$$
we can apply \eqref{harmonic-poincare} again to obtain
$$
\nr h-h(z)\nr_{2, \Om} \le \ol{\mu}(\Om)^{-1/2} \nr \na h\nr_{2,\Om}.
$$
\par
Thus, we can write, as in Theorem \ref{thm:W22-stability}, that 
$$
\nr h-h(z)\nr_{W^{2,2}(\Om)}\le \sqrt{1+\ol{\mu}(\Om)^{-1}+\ol{\mu}(\Om)^{-2}}\,\nr \na^2 h\nr_{2,\Om} ,
$$
and 
$$
\nr h-h(z)\nr_{2,\Om}\le \ol{\mu}(\Om)^{-1}\nr\na^2 h\nr_{2,\Om}.
$$
The rest of the proof runs similarly to that of Theorem \ref{thm:W22-stability}.
\end{proof}

\par
If $\Om$ is convex, the presence of the spectral quantity $\mu_0(\Om)$ in \eqref{W22-stability} and \eqref{W22-smallness} can be removed and replaced by a purely geometric quantity. This can be done by modifying Lemma \ref{lem:L2-estimate-oscillation}. In this case,
we know that the solution of \eqref{serrin1} has only one minimum point (see \cite{Ko}, for instance) and this can be joined to any boundary point by a segment.

\begin{lem}
\label{lem:L2-estimate-oscillation-convex}
Let $\Om\subset\RR^N$ be a convex domain. Let $z$ be the minimum point of 
the solution $u$ of \eqref{serrin1}. Set $h=q-u$, where $q$ is given by \eqref{quadratic}. 
\par
Then \eqref{W22-stability} holds, if \eqref{W22-smallness} is in force, with
$$
C=k_N \, \frac{ d_{\Om}^{\frac{4}{N+2}} \, M^{ \frac{N}{N+2} } }{ |\Om|^{ \frac{1}{N} } } 
\quad \mbox{ and } \quad
\ve=  \al_N \, \frac{M \, r_{i}^{\frac{N+2}{2}} }{d_{\Om}^2} \, . 
$$
\end{lem}

\begin{proof}
We begin by proceeding as in the proof of Lemma \ref{lem:L2-estimate-oscillation}. We let $x_i$ and $x_e$ be points in $\Ga$ that minimize (resp. maximize) $q$ on $\Ga$ and, for 
$$
0<\si<r_i,
$$ define the two points in $y_i, y_e\in\pa\Om_\si$ by
$y_j=x_j-\si\nu(x_j)$, $j=i, e$. 
As already done, we obtain the inequality:
$$
\frac12\,(\rho_e-\rho_i) (\rho_e+\rho_i-2 \si)\le h(y_e)-h(y_i)+2M \si 
$$
for $0<\si< \min \left\lbrace  \frac{\rho_e+\rho_i }{2}, r_i \right\rbrace  $.
\par
Since $\Om_\si$ is convex, we can join each $y_j$ to $z$ by a segment and, since $\na h(z)=0$, we can write the identity:
$$
h(y_j)-h(z)=\int_0^1 (1-t)\,\frac{d^2h}{dt^2}(z+t (y_j-z))\,dt.
$$
Thus,
$$
|h(y_j)-h(z)|\le |y_j-z|^2\,|\na^2 h(z_j)| \le \rho_j^2\,|\na^2 h(z_j)|,
$$
where $z_j$ is some point in $\Om_\si$.
\par
Then, we apply the mean value property to $|\na^2 h|$ (in fact this is subharmonic) in the ball $B_\si(z_j)$ and obtain as done before that
$$
|h(y_j)-h(z)|\le \frac{\rho_j^2}{\sqrt{|B|}\,\si^{N/2}}\,\nr \na^2 h\nr_{2,\Om}.
$$
Therefore, we find the inequality 
$$
\frac12\,(\rho_e-\rho_i) (\rho_e+\rho_i-2 \si)\le \frac{\rho_i^2+\rho_e^2}{\sqrt{|B|}\,\si^{N/2}}\,\nr \na^2 h\nr_{2,\Om}+2M \si 
$$
for $0<\si< \min \left\lbrace  \frac{ \rho_e+\rho_i }{2} , r_i \right\rbrace $.
\par
Observing that $\rho_i^2 + \rho_e^2 \le \frac{5}{4} d_{\Om}^2$ we get that
$$
\rho_e - \rho_i \le \frac{1}{\si_0} \left[  \frac{\frac{5}{4} d_{\Om}^2}{\sqrt{|B|}\,\si^{N/2}}\,\nr \na^2 h\nr_{2,\Om}+2M \si \right]
$$
and we finally conclude as done in the proof of Lemma \ref{lem:L2-estimate-oscillation}.
\end{proof}

\begin{rem}
\label{rem:spectral lower bounds}
{\rm
(i) It is clear that Lemma \ref{lem:L2-estimate-oscillation-convex} still holds in a domain for which we can claim that $y_e$ can be joined to $z$ by a segment. We stress the fact that, instead, the point $y_i$ can always be joined to $z$ by a segment.
\par
(ii) As observed in Remark \ref{rem:eigenvalues}, the value $\ol{\mu}(\Om)$ can be bounded below by
the second Neumann eigenvalue $\mu_2(\Om)$ that, in turn, can be estimated by geometrical parameters or isoperimetric constants. In the case that $\Om$ is convex, estimates involving the diameter $d_\Om$ alone can be found in \cite{PW} and \cite{ENT}. In the case of a general Lipschitz bounded domain, a lower bound for $\mu_2 (\Om)$ involving the best isoperimetric constant relative to $\Om$ can be found in \cite{BCT}.
\par
(iii) A lower bound for $\mu_0(\Om)$ can be obtained as follows. For any $B_r(x_0) \subset \Om$, by the mean value property we have that
$$
\mu_0(\Om)=\inf\Biggl\{\int_\Om |\na v|^2\,dx: \int_\Om v^2 dx=1, \De v=0 \mbox{ in } \Om, \int\limits_{B_r(x_0)} v \, dx =0 \Biggr\}
$$
and clearly,
$$
\mu_0(\Om) \ge \inf\Biggl\{\int_\Om |\na v|^2\,dx: v \in W^{1,2}(\Om), \int_\Om v^2 dx=1, \int\limits_{B_r(x_0)} v \, dx =0 \Biggr\}.
$$
Here, the right-hand side is the reciprocal of the optimal constant in the following Poincar\'e-type inequality considered in \cite{Me}[Theorem 1]:
$$
\int_{\Om} v^2 \, dx \le C\,\int_{\Om} |\na v|^2 \, dx,
$$
that holds with
$$
C=(1+r^{-N/2}\sqrt{|\Om|/|B|})^2\,(1+\mu_2(\Om)^{-2})-1,
$$
for every $v\in W^{1,2}(\Om)$ that has null mean value on $B_r(x_0)$ ($C$ has been computed by using \cite[Theorem 1]{Me} and \cite[Theorem 3.3 and Example 3.5]{AMR}).
It is clear that 
$\mu_0 (\Om)\ge 1/\sqrt{C}$.
Notice that we can always choose $r=r_i$.
}
\end{rem}

\medskip

We conclude this section by presenting a simple method to estimate the number $M$ in a quite general domain. The following lemma results from a simple inspection and by the uniqueness for the Dirichlet problem.

\begin{lem}[Torsional creep in an annulus]\label{soluzioneincoronacircolare}
\label{lem:torsion-annulus}

Let $A=A_{r,R}\subset \RR^N$ be the annulus centered at the  origin and radii $0<r<R$,
and set $\ka=r/R$.
\par
Then, the solution $w$ of the Dirichlet problem 
\begin{equation*}
\label{torsion-annulus}
\Delta w = N \ \textrm{ in } \ A, \quad w = 0 \ \textrm{ on } \ \pa A,
\end{equation*}
is defined for $r\le |x|\le R$ by
\begin{equation*}
w(x) =
\begin{cases}
\displaystyle\frac12\, |x|^2 +\frac{R^2}{2}\,(1-\ka^2)\,\frac{\log(|x|/r)}{\log\ka} -\frac{r^2}{2} 
\ &\mbox{ for } \ N=2, 
\vspace{5pt} \\
\displaystyle\frac12\,|x|^2 +\frac12\,\frac{R^2}{1-\ka^{N-2}}\,\left\{ (1-\ka^2)\,(|x|/r)^{2-N}+\ka^N-1\right\} \ &\mbox{ for } \ N \ge 3.
\end{cases}
\end{equation*}
\end{lem}

\medskip

\begin{thm}[A bound for the gradient on $\Ga$]
\label{thm:boundary-gradient}
Let $\Om\subset\RR^N$ be a bounded domain that satisfies the uniform interior and exterior conditions with radii $r_i$ and $r_e$ and let $u\in C^1(\ol{\Om})\cap C^2(\Om)$ be a solution
of \eqref{serrin1} in $\Om$.
\par
Then, we have that
\begin{equation}
\label{gradient-estimate}
r_i\le |\na u| \le 
c_N\,\frac{d_\Om(d_\Om+r_e)}{r_e} \ \mbox{ on } \ \Ga,
\end{equation}
where $d_\Om$ is the diameter of $\Om$ and 
$c_N=3/2$ for $N=2$ and $c_N=N/2$ for $N\ge 3$.
\end{thm}

\begin{proof}
We first prove the first inequality in \eqref{gradient-estimate}. Fix any  $p\in\Ga$.
Let $B=B_{r_i}$ be the interior ball touching $\Ga$ at $p$ and place the origin of cartesian axes at the center of $B$.
\par
If $w$ is the solution of \eqref{serrin1} in $B$, that is $w(x)= (|x|^2-r_i^2)/2$, by comparison we have that $w\ge u$ on $\ol{\Om}$ and hence, since $u(p)=w(p)=0$, we obtain:
$$
u_\nu(p)\ge w_\nu(p)=r_i.
$$
\par
To prove the second inequality, we place the origin of axes at the center of  the exterior ball $B=B_{r_e}$ touching $\Ga$ at $p$. Denote by $A$ the smallest annulus containing $\Om$, concentric with $B$ and having $\pa B$ as internal boundary and let $R$ be the radius of its external boundary.
\par
If $w$ is the solution of \eqref{serrin1} in $A$, by comparison we have that $w\le u$ on $\ol{\Om}$.
Moreover, since $u(p)=w(p)=0$, we have that
$$
u_\nu(p)\le w_\nu(p).
$$
By Lemma \ref{lem:torsion-annulus} we then compute that
$$
w_\nu(p) =\frac{R (R-r_e)}{r_e}\,f(\ka)
$$
where, for $0<\ka<1$,
\begin{equation}
\label{def-f}
f(\ka)=
\begin{cases}
\displaystyle \frac{2\ka^2\log(1/\ka)+\ka^2-1}{2(1-\ka) \log(1/\ka)}
\ &\mbox{ for } \ N=2, 
\vspace{5pt} \\
\displaystyle\frac{2 \ka^N-N \ka^2+N-2}{2(1-\ka)(1-\ka^{N-2})} \ &\mbox{ for } \ N \ge 3.
\end{cases}
\end{equation}
Notice that $f$ is bounded since it can be extended to a continuous function on $[0,1]$.
Tedious calculations yield that 
$$
\sup_{0<\ka<1} f(\ka)=
\begin{cases}
\frac32 \ &\mbox{ for } \ N=2, \\
\frac{N}2 \ &\mbox{ for } \ N\ge 3.
\end{cases}
$$
Finally, observe that $R\le d_\Om+r_e$.
\end{proof}

\begin{rem}
{\rm
To the best of our knowlwdge, inequality \eqref{gradient-estimate} is not 
present in the literature for general smooth domains and is not sharp.  
Other estimates are given in \cite{PP} for planar strictly convex domains (but the same argument can be generalized to general dimension for strictly mean convex domains) and in \cite{CM} for strictly mean convex domains in genearal dimension.
In particular, in \cite[Lemma 2.2]{CM} the authors prove that there exists a universal constant $c_0$ such that
\begin{equation}
\label{NUOVA maggiorgrad}
| \na u | \leq c_0 | \Om |^{1/N} \mbox{ in } \ol{\Om}.
\end{equation}
 
\par
Since the focus of this paper is not on the sharpness of constants, we chose to present the elementary proof of Theorem \ref{thm:boundary-gradient}.
}
\end{rem}

\section{Stability for the Soap Bubble Theorem \\
and some overdetermined problems}
\label{sec:stability}

In this section, we collect our results on the stability of the spherical configuration by putting together the identities derived in Section \ref{sec:SBT} and the estimates obtained in Section \ref{sec:estimates}.
\par
It is clear that, we may replace $\nr H_0-H\nr_{1,\Ga}$ by  the weaker deviation
$$
\int_\Ga (H_0-H)^+\,dS_x,
$$
in all the relevant formulas in the sequel.
\par
We begin with our main result.
\begin{thm}[General stability for the Soap Bubble Theorem]
\label{thm:SBT-stability}
Let $\Ga$ be the connected boundary of class $C^{2,\al}$, $0<\al<1$, of a bounded domain $\Om\subset\RR^N$, $N\ge 2$. Denote by $H$ its mean curvature function and let $H_0$
be the constant defined in \eqref{R and H_0}.
\par
There is a point $z\in\Om$ such that
\begin{enumerate}[(i)]
\item
if $N=2$ or $N=3$,  there exixts a positive constant $C$ such that
\begin{equation}
\label{general-stability-Lipschitz}
\rho_e-\rho_i\le C\,\nr H_0-H\nr_{1,\Ga}^{1/2};
\end{equation}
\item
If $N\ge 4$, there exist two positive constants $C$ and $\ve$ such that \begin{equation}
\label{general-stability}
\rho_e-\rho_i\le C\,\nr H_0-H\nr_{1,\Ga}^{1/(N+2)} \quad \mbox{ if } \quad \nr H_0-H\nr_{1,\Ga}<\ve.
\end{equation}
\end{enumerate}
\par
The constants $C$ and $\ve$ depend on the dimension $N$, the geometrical quantities $|\Om|$, $d_\Om$, $r_e$, $r_i$, the spectral parameter $\mu_0(\Om)$ defined in \eqref{inf-harmonic-poincare2} and in the case (i) also on the immersion constant $c$ introduced in Theorem \ref{thm:W22-stability}. Their explicit expression are given in \eqref{def-C-Lipschitz} and \eqref{def-C-eps}.
\end{thm}

\begin{proof}
Let $u$ be the solution of \eqref{serrin1} and let $z\in\Om$ be any local minimum point of $u$ in $\Om$.
Set $h=q-u$, where $q$ is given by \eqref{quadratic}. From \eqref{H-fundamental} and \eqref{L2-norm-hessian}, we infer that
\begin{equation}
\label{provaperosservazione}
\nr \na^2 h\nr_{2,\Om}\le M\,\sqrt{N-1}\,\nr H-H_0\nr_{1,\Ga}^{1/2}
\end{equation}
\par
If $N=2$ or $N=3$,  by \eqref{Lipschitz-stability} and \eqref{gradient-estimate} we obtain \eqref{general-stability-Lipschitz}
at once with
\begin{equation}
\label{def-C-Lipschitz}
C= k_N \, c \, \frac{d_\Om^\ga}{|\Om|^{ \frac{1}{N} }} \, \frac{1+\mu_0(\Om)}{\mu_0(\Om)} \, \frac{d_{\Om} (d_{\Om} +r_e)}{r_e} .
\end{equation}
 \par
When $N\ge 4$, if \eqref{W22-smallness} holds, then \eqref{W22-stability} informs us that
\begin{multline*}
\rho_e-\rho_i\le \frac{a_N\, M^{N/(N+2)}}{\mu_0(\Om)^{2/(N+2)}|\Om|^{1/N}}\, \nr \na^2 h\nr_{2,\Om}^{2/(N+2)}\le \\
\frac{a_N\, (N-1)^{1/(N+2)}}{\mu_0(\Om)^{2/(N+2)}|\Om|^{1/N}}\, M\,\nr H-H_0\nr_{1,\Ga}^{1/(N+2)} ,
\end{multline*}
where $a_N$ is the constant defined in \eqref{provacondefcostaN}.
Thus, there are constants $k_N$ and $\al_N$ such that \eqref{general-stability} holds with
\begin{equation}
\label{def-C-eps}
C= k_N\, \frac{d_\Om(d_\Om+r_e)}{ \mu_0(\Om)^{ \frac{2}{N+2} }|\Om|^{ \frac{1}{N} } \,r_e} \quad \mbox{ and } \quad \ve= \al_N \, \mu_0(\Om)^2 \, r_{i}^{N+2},
\end{equation}
by \eqref{gradient-estimate}.
\end{proof}

\begin{rem}
\label{rem:SBT-stability-neumann}
{\rm
(i) The distance of a minimum point of $u$ from $\Ga$ may be estimated from below, in terms of geometrical and spectral parameters, by following the arguments contained in \cite{BMS}. 
\par
(ii)
Another version of Theorem \ref{thm:SBT-stability} can be stated if we assume that $\Om$ contains its center of mass. The proof runs similarly. In fact, it suffices to use Theorem \ref{th:W22-stability-neumann} instead of Theorem \ref{thm:W22-stability}. In this way, we simply obtain the constants given in \eqref{def-C-Lipschitz} and \eqref{def-C-eps}, with $\mu_{0}(\Om)$ replaced by $\ol{\mu}(\Om)$. 
Remark \ref{rem:eigenvalues} then informs us that such constants are slightly better.
\par
(iii) In \eqref{general-stability}, the assumption that $\nr H_0-H\nr_{1,\Ga}<\ve$ may leave the impression that (ii) of Theorem \ref{thm:SBT-stability} is not a \textit{global} stability result. However, if $\nr H_0-H\nr_{1,\Ga}\ge\ve$, it is a trivial matter to obtain an upper bound for $\rho_e - \rho_i$ in terms of $\nr H_0-H\nr_{1,\Ga}$.
}
\end{rem}

Since the estimate in Theorem \ref{thm:SBT-stability} does not depend on the particular minimum point chosen, as a corollary, we obtain a result of closeness to a union of balls.

\begin{cor}[Closeness to an aggregate of balls]
\label{cor:SBT-stability}
Let $\Ga$, $H$, and $H_0$ be as in Theorem \ref{thm:SBT-stability}.
\par
Then, there exist points $z_1, \dots, z_n$ in $\Om$, $n\ge 1$, and corresponding numbers
\begin{equation}
\label{def-rho_j}
\rho_i^j=\min_{x\in\Ga}|x-z_j| \ \mbox{ and } \ \rho_e^j=\min_{x\in\Ga}|x-z_j|,
\quad j=1, \dots, n,
\end{equation}
such that
\begin{equation}
\label{aggregate}
\bigcup_{j=1}^n B_{\rho_i^j}(z_j)\subset\Om\subset \bigcap_{j=1}^n B_{\rho_e^j}(z_j)
\end{equation}
and
$$
\max_{1\le j\le n}(\rho_e^j-\rho_i^j)\le C\,\nr H_0-H\nr_{1,\Ga}^{1/2},
$$
if $N=2$ or $N=3$, and
$$
\max_{1\le j\le n}(\rho_e^j-\rho_i^j)\le C\,\nr H_0-H\nr_{1,\Ga}^{1/(N+2)} 
\quad \mbox{ if } \quad \nr H_0-H\nr_{1,\Ga}<\ve,
$$
if $N\ge 4$. Here, the relevant constants are those in \eqref{def-C-Lipschitz}  and \eqref{def-C-eps}.
\par
The number $n$ can be chosen as the number of connected components of the set $\cM$ of all the local minimum points of the solution $u$ of \eqref{serrin1}.
\end{cor}

\begin{proof}
Let $\cM_j$, $j=1,\dots, n$, be the connected components of $\cM$ and pick one point $z_j$ from
each $\cM_j$. By applying Theorem \ref{thm:SBT-stability} to each $z_j$, the conclusion is then evident.
\end{proof}

\begin{rem}
{\rm

The estimates presented in Theorem \ref{thm:SBT-stability} and sketched in (ii) of Remark \ref{rem:SBT-stability-neumann}, may be interpreted as stability estimates, once some a priori information is available: here, we just illustrate the case (ii) of Theorem \ref{thm:SBT-stability}. 
Given four positive constants $d, r, V,$ and $\mu$, let $\cS=\cS(d, r, V,\mu)$ be the class of connected surfaces $\Ga\subset\RR^N$ of class $C^{2,\al}$, where $\Ga$ is the boundary of a bounded domain $\Om$, such that 
$$
d_{ \Om} \le d, \quad r_i(\Om), r_e(\Om)\ge r, \quad |\Om|\ge V, \quad \mu_0(\Om)\ge\mu.
$$
Then, for every $\Ga\in\cS$ with $\nr H_0-H\nr_{1,\Ga}< \ve $, we have that
$$
\rho_e-\rho_i\le C\, \nr H_0-H\nr_{1,\Ga}^{1/(N+2)},
$$
where $C$ and $\ve$ are the constants in \eqref{def-C-eps}, with the relevant parameters replaced by 
the constants $d, r, V,\mu$.
\par
If we relax the a priori assumption that $\Ga\in\cS$, it may happen that, 
as the deviation $\nr H_0 - H\nr_{1,\Ga}$ tends to $0$, $\Om$ tends to the ideal configuration of two or more mutually tangent balls, while $\ve$ tends to $0$ and $C$ diverges since $r$ tends to $0$. This behavior can be avoided by considering strictly mean convex surfaces, as done in \cite{CM} by using the uniform deviation $\nr H_0 - H\nr_{\infty,\Ga}$.
}
\end{rem}

\medskip

If we suppose that $\Ga$ is strictly mean convex, then we can use Theorem \ref{th:heintze-karcher} to 
to obtain a stability result for Heintze-Karcher inequality and we can improve the constants $C$ and $\ve$ in \eqref{general-stability}.

\begin{thm}[Stability for Heintze-Karcher's inequality]
\label{thm:stability-hk}
Let $\Ga$ be the connected boundary of class $C^{2,\al}$, $0<\al<1$, of a bounded domain $\Om\subset\RR^N$, $N\ge 2$. Denote by $H$ its mean curvature function and suppose that $H>0$ on $\Ga$.
\par
There is a point $z\in\Om$ such that
\begin{enumerate}[(i)]
\item
if $N=2$ or $N=3$,  there exixts a positive constant $C$ such that
\begin{equation}
\label{stability-hk-Lip23}
\rho_e-\rho_i\le C\, \left(\int_\Ga\frac{dS_x}{H}-N\,|\Om|\right)^{1/2} ;
\end{equation}
\item
If $N\ge 4$, there exist two positive constants $C$ and $\ve$ such that
\begin{equation}
\label{stability-hk}
\rho_e-\rho_i\le C\,\left(\int_\Ga\frac{dS_x}{H}-N\,|\Om|\right)^{1/(N+2)},
\end{equation}
if 
$$
\left( \int_\Ga\frac{dS_x}{H}-N\,|\Om| \right) <\ve.
$$
\end{enumerate}
\par
The relevant constants will be given in \eqref{def-C-hk-Lipschitz} and \eqref{def-C-eps-hk}.
\end{thm}

\begin{proof}
We chose the point $z$ in $\Om$ as in the proof of Theorem \ref{thm:SBT-stability}.
Moreover, by \eqref{heintze-karcher-identity} and \eqref{heintze-karcher}, we have that 
$$
\frac1{N-1}\,\int_\Om |\na^2 h|^2\,dx\le \int_\Ga\frac{dS_x}{H}-N\,|\Om|.
$$
We then proceed as in the proof of Theorem \ref{thm:SBT-stability} and obtain \eqref{stability-hk-Lip23} with
\begin{equation}
\label{def-C-hk-Lipschitz}
C= k_N \, c \, \frac{d_\Om^\ga}{|\Om|^{ \frac{1}{N} }} \, \frac{1+\mu_0(\Om)}{\mu_0(\Om)} ,
\end{equation}
if $N=2$ or $N=3$, with the help of \eqref{Lipschitz-stability}, and \eqref{stability-hk} with
\begin{equation}
\label{def-C-eps-hk}
C= k_N \, \frac{ M^{ \frac{N}{N+2} } }{ \mu_0(\Om)^{ \frac{2}{N+2} } \, |\Om|^{ \frac{1}{N} } } 
\quad \mbox{ and }  \quad
\ve= \al_N \, \mu_0(\Om)^2 \, M^2 \, r_{i}^{N+2} ,
\end{equation}
with the help of \eqref{W22-stability}, \eqref{W22-smallness}.

To avoid the presence of $M$ in the constants $C$ and $\ve$, we can use respectively \eqref{NUOVA maggiorgrad} (obviously we could also use again the second inequality in \eqref{gradient-estimate} as before) and the first inequality in \eqref{gradient-estimate} and choose
\begin{equation*}
C= k_N \, \mu_0(\Om)^{-\frac{2}{N+2} } \, |\Om|^{-\frac{2}{N(N+2)} } \quad \mbox{ and } \quad
\ve = \al_N \, \mu_0(\Om)^2 \, r_{i}^{N+4}.
\end{equation*}
\end{proof}

The following theorem is in the spirit of the main result contained in \cite{CV} (see also \cite{CM}).

\begin{thm}[Stability for strictly mean convex hypersurfaces]
\label{thm:SBT-stability-mean-convex}
Let $\Ga$ be the connected boundary of class $C^{2,\al}$, $0<\al<1$, of a bounded domain $\Om\subset\RR^N$, $N\ge 2$. Denote by $H$ its mean curvature function, suppose that there exists a constant $\ul{H}>0$ such that $H\ge\ul{H}$ on $\Ga$, and let $H_0$ be the constant defined in \eqref{R and H_0}.
\par
There is a point $z\in\Om$ such that
\begin{enumerate}[(i)]
\item
if $N=2$ or $N=3$,  there exixts a positive constant $C$ such that
\begin{equation}
\label{SBT-stability-mean-convex-Lip23}
\rho_e-\rho_i\le C\, \nr H_0-H\nr_{\infty,\Ga}^{1/2} ;
\end{equation}
\item
If $N\ge 4$, there exist two positive constants $C$ and $\ve$ such that
\begin{equation}
\label{SBT-stability-mean-convex}
\rho_e-\rho_i\le C\,\nr H_0-H\nr_{\infty,\Ga}^{1/(N+2)},
\end{equation}
if 
$$
\nr H_0-H\nr_{\infty,\Ga}<\ve .
$$
\end{enumerate}
\par
The  relevant constants will be given in \eqref{def-C-Lip-mc} \eqref{def-C-eps-mc}.
\end{thm}
\begin{proof}
We simply observe that 
$$
\int_\Ga\frac{dS_x}{H}-N\,|\Om|=\int_\Ga\left[\frac1{H}-\frac1{H_0}\right]\,dS_x\le
\frac{N |\Om|}{|\Ga|}\,\nr H_0-H\nr_{\infty,\Ga} \int_\Ga\frac{dS_x}{H},
$$
and hence from \eqref{heintze-karcher-identity} and the fact that $H \ge \ul{H}$ on $\Ga$ it follows that
$$
\frac1{N-1}\,\int_\Om |\na^2 h|^2\,dx\le \frac{N |\Om|}{\ul{H}}\,\nr H_0-H\nr_{\infty,\Ga}.
$$

The rest of the proof runs similarly to those of Theorems \ref{thm:SBT-stability} and \ref{thm:stability-hk}.

If $N=2$ or $N=3$ we obtain \eqref{SBT-stability-mean-convex-Lip23} with
\begin{equation}
\label{def-C-Lip-mc}
C= k_N \, c \, \frac{ d_\Om^\ga \, |\Om|^{\frac{1}{2} - \frac{1}{N} } }{ \ul{H}^{\frac{1}{2} } } \, \frac{1+\mu_0(\Om)}{\mu_0(\Om)} .
\end{equation}
If $N \ge 4$ we obtain \eqref{SBT-stability-mean-convex} with
\begin{equation}
\label{def-C-eps-mc}
C= k_N \, \frac{ M^{ \frac{N}{N+2} } }{\mu_0 (\Om)^{ \frac{2}{N+2} } \, |\Om|^{ \frac{1}{N} - \frac{1}{N + 2} } \, \ul{H}^{ \frac{1}{N+2} } } 
\quad\mbox{ and }
\quad
\ve= \al_N \, \frac{ \ul{H} }{|\Om|} \, \mu_0(\Om)^2 \, M^2 \, r_{i}^{N+2} .
\end{equation}
\par
As before, the presence of $M$ in $C$ and $\ve$ can be avoided by means of \eqref{NUOVA maggiorgrad} and the first inequality in \eqref{gradient-estimate}. 
\end{proof}

\begin{rem}
{\rm
(i) In Theorem \ref{thm:SBT-stability-mean-convex}, if the deviation $\nr H_0-H\nr_{\infty,\Ga}$ is small enough, $\ul{H}$ can be replaced by a fraction of $H_0$. 
Also, from the proof of that theorem, it is evident that the norm $\nr H_0-H\nr_{\infty,\Ga}$ can be replaced 
by the weaker one $\nr H_0-H\nr_{1,\Ga}$. 
\par
(ii) When $\Om$ is convex, by using Lemma \ref{lem:L2-estimate-oscillation-convex} instead of Theorem \ref{thm:W22-stability}, we can avoid the use of the spectral parameter $\mu_0 (\Om)$ in the constants of Theorems \ref{thm:SBT-stability}, \ref{thm:stability-hk}, \ref{thm:SBT-stability-mean-convex} and Corollary \ref{cor:SBT-stability}.
}
\end{rem}

\medskip

 The inequalities of Section \ref{sec:estimates} can also be used to obtain stability estimates for one of the two overdetermined boundary value problems mentioned in Section \ref{sec:SBT}.
 
\begin{thm}[Stability for an overdetermined problem]
\label{thm:OBVP-stability}
Let $\Ga$ and $\Om$ be as in Theorem \ref{thm:SBT-stability} and suppose that $H>0$ on $\Ga$.
\par
There is a point $z\in\Om$ such that
\begin{enumerate}[(i)]
\item
if $N=2$ or $N=3$,  there exixts a positive constant $C$ such that
\begin{equation}
\label{OBVT-stability-Lip23}
\rho_e-\rho_i\le C\,\nr u_\nu-1/H\nr_{1,\Ga}^{1/2} ;
\end{equation}
\item
If $N\ge 4$, there exist two positive constants $C$ and $\ve$ such that
\begin{equation}
\label{OBVT-stability-gener}
\rho_e-\rho_i\le C\,\nr u_\nu-1/H\nr_{1,\Ga}^{\frac1{N+2}},
\end{equation}
if 
$$
\nr u_\nu-1/H\nr_{1,\Ga}<\ve .
$$
\end{enumerate}
\par
The relevant constants will be given in \eqref{def-C-OBVP-Lip23} and \eqref{def-C-eps-OBVP}.
\end{thm}

\begin{proof}
We observe that 
$$
\int_\Ga (1-H\,u_\nu)\,u_\nu\,dS_x\le \int_\Ga |u_\nu-1/H| |H u_\nu|\,dS_x\le
\frac{M}{r_i}\,\nr u_\nu-1/H\nr_{1,\Ga},
$$
since $H\le 1/r_i$. 
Thus, by \eqref{fundamental-identity2} and \eqref{L2-norm-hessian} we have that
$$
\nr \na^2 h\nr_{2,\Om}^2 \le (N-1) \, \frac{M}{r_i} \, \nr u_\nu-1/H\nr_{1,\Ga} .
$$
By proceeding as before, we get \eqref{OBVT-stability-Lip23} with
\begin{equation}
\label{def-C-OBVP-Lip23}
C = k_N \, c \, \frac{ d_\Om^\ga}{|\Om|^{ \frac{1}{N} }} \, \frac{1+\mu_0(\Om)}{\mu_0(\Om)} \, \sqrt{ \frac{ M }{ r_i } } 
\end{equation}
if $N=2$ or $N=3$, and \eqref{OBVT-stability-gener} with

\begin{equation}
\label{def-C-eps-OBVP}
C = k_N \, \frac{ M^{\frac{N+1}{N+2} } }{ \mu_0(\Om)^{ \frac{2}{N+2} }|\Om|^{ \frac{1}{N} } \, r_{i}^{ \frac{1}{N+2} } }  \ \mbox{ and } \ \ve= \al_N \, \mu_0(\Om)^2 \, M \, r_{i}^{N+3} ,
\end{equation}
if $N \ge 4$.
As before, by \eqref{NUOVA maggiorgrad} and the first inequality in \eqref{gradient-estimate}, we can replace $M$ in $C$ and $\ve$ in \eqref{def-C-OBVP-Lip23} and \eqref{def-C-eps-OBVP}. 
\end{proof}

\begin{rem}
{\rm
It is clear that estimates in the spirit of Corollary \ref{cor:SBT-stability} can also be given for the situations treated in Theorems \ref{thm:stability-hk}, \ref{thm:SBT-stability-mean-convex} and \ref{thm:OBVP-stability}.
}
\end{rem}

\section*{Acknowledgements}
The authors wish to thank prof. S. Sakaguchi (Tohoku University) for bringing up to their attention reference \cite{Re} and for many fruitful discussions.
\par
Remarks \ref{rem:spectral lower bounds} (iii) and \ref{rem:SBT-stability-neumann} (iii) were suggested by the anonymous referee. The authors warmly thank him/her for the nice improvements to this paper.
\par
The paper was partially supported by a grant iFUND-Azione 2 of the Universit\`a di Firenze, under a scientific and cultural agreement with Tohoku University, and by the GNAMPA (first author) and GNSAGA (second author) of the Istituto Nazionale di Alta Matematica (INdAM).

\end{document}